\documentclass[final]{siamart1116}

\usepackage{amsmath}
\usepackage{amssymb}
\usepackage{epsfig}
\usepackage{labelfig}

\newtheorem{remark}{Remark}[section]
\usepackage{lipsum}
\usepackage{amsfonts}
\usepackage{graphicx}
\usepackage{epstopdf}
\usepackage{algorithmic}
\ifpdf
  \DeclareGraphicsExtensions{.eps,.pdf,.png,.jpg}
\else
  \DeclareGraphicsExtensions{.eps}
\fi

\newcommand{\TheTitle}{Adaptive coarse space for domain decomposition methods for a $P_1$ discontinuous Galerkin multiscale problem}

\newcommand{\TheAuthors}{E. Eikeland, L. Marcinkowski, and T. Rahman}

\headers{Adaptively enriched coarse space for the DG multiscale}{\TheAuthors}

\title{{\TheTitle}\thanks{This work has partly been supported by the Norwegian Research Council and .}}

\author{
  Erik Eikeland\thanks{Department of Computing, Mathematics and Physics, Western Norway University of Applied Sciences, Norway
    (\email{eeik@hvl.no}, \email{tra@hvl.no}).}
  \and
  Leszek Marcinkowski \thanks{Department of Mathematics, Warsaw University, Poland (\email{lmarcin@mimuw.edu.pl}).}
  \and
  Talal Rahman \footnotemark[2]
}

\usepackage{amsopn}

\ifpdf
\hypersetup{
  pdftitle={\TheTitle},
  pdfauthor={\TheAuthors}
}
\fi

\begin{document}

\maketitle

% REQUIRED
\begin{abstract}
In this paper, we propose a two-level overlapping additive Schwarz domain decomposition preconditioner for the symmetric interior penalty discontinuous Galerkin method for the second order elliptic boundary value problem with highly heterogeneous coefficients. A specific feature of this preconditioner is that it is based on adaptively enriching its coarse space with functions created through solving generalized eigenvalue problems on thin patches covering the subdomain interfaces. 
It is shown that the condition number of the underlined preconditioned system is independent of the contrast if an adequate number of functions are used to enrich the coarse space. Numerical results are provided to confirm this claim.
\end{abstract}

% REQUIRED

% REQUIRED

\section{Introduction} \label{sec:intro}
We consider the symmetric interior penalty discontinuous Galerkin (SIPG) discretization of the second order elliptic boundary value problem with highly heterogeneous coefficients representing, for instance, the permeability of a porous media in reservoir simulation, and we propose a new additive Schwarz preconditioner for the iterative solution of the resulting system. It is already known in the community that standard domain decomposition methods, in general, have difficulties robustly dealing with the heterogeneity, particularly when they vary highly along subdomain boundaries affecting the Poincar\'e inequality (constant), cf. e.g., \cite{Pechstein:2012:WPI}. This may be explained by saying that a standard coarse space is not rich enough to capture all the worst modes in the residual, and therefore requires some form of enrichment, see for instance \cite{Galvis:2010:DDM,Galvis:2010:DDMR,Efendiev:2012:RDD,Dolean:2012:ATL,Loneland:2015:PEB} for some work on adaptively enriching the coarse space in the continuous Galerkin case. The objective of this paper is to further extend this idea to the discontinuous Galerkin case, by proposing a new and effective coarse problem for the discontinuous Galerkin method, which is based on adaptively enriching the coarse space with functions corresponding to the bad eigenmodes of certain appropriate generalized eigenvalue problems, thereby removing their adverse effect on the convergence.        

The SIPG is a symmetric version of the discontinuous Galerkin (DG), a methodology, which has in the recent years become increasingly popular in the scientific computing community. In contrast to the classical conforming and nonconforming techniques, the DG methods allow for the finite element functions to be entirely discontinuous across the element boundaries, thereby allowing for additional flexibilities with regards to using irregular meshes, local mesh refinement, and different polynomial degrees for the basis functions on different elements, cf. e.g., \cite{Cockburn:2000:DGM,Arnold:2002:UAD,Brezzi:2006:SMD,Riviere:2008:DGM}. DG methods are also preferred when dealing with models based on the laws of conservation, because such methods are locally mass conservative, whereas the classical methods only preserve a global mass balance. For an introduction to the DG methodology see, e.g., \cite{Riviere:2008:DGM}, and for an overview of recent developments in the field we refer to, e.g., \cite{DiPietro:2012:MADG,Karakashian:2013:RDD,Feistauer:2015:DGM} and the references therein. The systems resulting from the SIPG discretization are symmetric and positive definite, and in general very large. Krylov iterations, like the conjugate gradients (CG), are therefore used for the solution of such systems, together with appropriate preconditioners for improved convergence. 

Additive Schwarz preconditioners are among the most popular preconditioners based on the domain decomposition, which are inherently parallel and are easy to implement, cf. \cite{Toselli:2005:DDM, Mathew:2008:DDM,Smith:1996:DDP, Quarteroni:1999:DDM}. However, for most domain decomposition algorithms for the DG that exist to date, it is assumed that the coefficients are either constants or piece wise constants with respect to some partition of the domain, that is, in the latter case, the coefficients may be constant, or mildly varying, inside each subdomain, and have jumps only across subdomain boundaries, cf. e.g. \cite{Gopal:2003:MLDG,Feng:2006:TLASMDG,Dryja:2007:BDDDCMDG,Antonietti:2012:SMP,Brenner:2013:BDD,Dryja:2014:AFETIDPDG,Antonietti:2014:DDP,Antonietti:2015:SPHP,Antonietti:2016:NOS} and references therein. In case of multiscale problems, where the coefficients may vary rapidly and everywhere, in particular when the coefficients vary along subdomain boundaries, such methods are not robust in general; even with what we know as the multiscale finite element coarse space, cf. e.g. \cite{Ma:2012:DDPDG,Du:2014:FETIDPDG} in case of the discontinuous Galerkin, and \cite{Graham:2007:DDM,Pechstein:2008:AFM,Pechstein:2013:FBE} in case of the continuous Galerkin, the condition number still depends on the variation (heterogeneity).
The preconditioner proposed in this paper is based on the abstract Schwarz framework, where nonoverlapping subdomains are used for the local subproblems, and an adaptively enriched coarse space for the coarse problem to get an algorithm that is robust with respect to the variation (heterogeneity). 
Starting with a standard multiscale finite element coarse space, the coarse space is enriched with functions built through solving a set of generalized eigenvalue problems on a set of thin patches each covering a subdomain interface, and including those functions that correspond to the eigenvalues that are below a given threshold. 

Adaptive enrichment of coarse spaces with functions generated through solving certain eigenvalue problems for the construction of robust domain decomposition methods, has started to attract much interest in the recent years, cf. e.g. \cite{Galvis:2010:DDM,Nataf:2010:TLDD,Nataf:2011:CSC,Efendiev:2012:RTD,Dolean:2012:ATL,Spilane:2014:ARC,Loneland:2015:ANH,Dolean:2015:IDD,Eikeland:2016:OSM} for recent applications of the approach to second order elliptic problems with heterogeneous coefficient. There are also results for FETI-DP and BDDC substructuring domain decomposition methods where similar ideas have been used for constructing the coarse space (the primal constraints), cf. \cite{mandel2007adaptive,sousedik2013adaptive,kim2015bddc,KRR:2015:FDMACS,klawonn2016comparison} in 2D and \cite{calvo2015adaptive,kim2016bddc,klawonn:2016:adaptive,oh:2016:BDDC,pechstein2016unified} in 3D.  

The present work is a step in the same direction and is based on the idea of solving lower dimensional eigenvalue problems, that is 1D eigenvalue problems in the 2D, and 2D eigenvalue problems in the 3D, and then appropriately extending the eigenfunctions inside to be included into the coarse space. This idea was first proposed for the additive Schwarz method in \cite{Loneland:2015:ANH} for 2D problems, and later extended to 3D problems in \cite{Eikeland:2016:OSM}, using continuous and piecewise linear finite elements. The present work is an extension of the idea to the discontinuous Galerkin case, proposing a new multiscale finite element space for the DG method, inspired by the work in \cite{Graham:2007:DDM}, and solving a set of relatively small sized eigenvalue problems on thin patches covering the subdomain interfaces in order to deal with the discontinuity of the DG functions. The resulting algorithm is both robust and cost-effective.

The rest of the paper is organized as follows. The discrete formulation of the problem based on the SIPG is given in section~\ref{sec:Discrete}; in section~\ref{sec:Prec} we introduce the new additive Schwarz preconditioner for the discrete problem. The convergence analysis is given in section~\ref{sec:Proofs}, while the numerical results are given in section~\ref{sec:NumRes}. 

Throughout this paper, we use the following notations: $x\lesssim y$ and $w\gtrsim z$ denote that there exist positive constants $c,C$ independent of mesh parameters and the jump of coefficients such that $x\leq c y$ and $w\geq C z$, respectively.

\section{Discrete problem}\label{sec:Discrete}
We consider the following linear variational problem: Find $u^*\in H_0^1(\Omega)$ such that
\begin{equation} \label{eq:modelproblem}
 a(u^*,v)=f(v) \quad \forall v \in H^1_0(\Omega),
\end{equation}
where $\Omega$ is a polygonal domain on the plane and  $a(u,v)=\int_\Omega \alpha \nabla u\cdot \nabla v \: d x$ is a symmetric and positive definite bilinear form for all $u, v \in H^1_0(\Omega)$. Here $\alpha\in L^{\infty}(\Omega)$ is a positive function, and there exists $C_0>0$ such that $\alpha(x) \geq C_0$. For simplicity, since $\alpha$ can be scaled by $C_0^{-1}$, we assume that $\alpha\geq 1$, which, however, may be highly varying.  

We assume that there exists a sequence of  quasi-uniform triangulations (cf. \cite{Brenner:2008:MTF,Braess:2007:FE}) of $\Omega$, $\mathcal{T}_h=\mathcal{T}_h(\Omega)=\{\tau\}$, consisting of triangles with the parameter $h=\max_{\tau \in {\cal T}_h} diam(\tau)$ tending to zero. 

 Since we consider only piecewise linear FEM functions, see below,  we assume that $\alpha$ is constant on each $\tau\in {\cal T}_h$,  without any loss of generality. 

Let $\mathcal{E}_h$ denote the set of all edges of the triangles of ${\cal T}_h$, which can be split into two subsets, $\mathcal{E}_h^\partial$, the set consisting of edges that are on the boundary $\partial\Omega$, and $\mathcal{E}_h^0$, the set consisting of edges that are in the interior of $\Omega$. Let ${\cal V}_h$ denote the set of all vertices of the triangles of ${\cal T}_h$, and ${\cal V}_h(\tau)$ the set of vertices of the triangle $\tau$. 

For each edge $e\in \mathcal{E}_h^0$, the common edge between two neighboring triangles $\tau_+$ and $\tau_-$, we introduce the following weights,  cf. \cite{Dryja:2003:ODG,Burman:2006:DDM,Z.Cai:2011:DGF},
\begin{eqnarray*}
\omega^e_+=\alpha_-/(\alpha_++\alpha_-) \quad\mbox{and}\quad \omega^e_-=\alpha_+/(\alpha_++\alpha_-),
\end{eqnarray*}
where $\alpha_+$ and $\alpha_-$ are the restrictions of $\alpha$ to $\tau_+$ and $\tau_-$, respectively. Note that
 \begin{eqnarray*}
   \omega_+^e + \omega_-^e = 1.
 \end{eqnarray*} 
Further, we introduce the following notations. For each $e\subset \mathcal{E}_h^0$, define
\begin{equation}
[u] = u_+\: n_+ + u_- \: n_- \quad\mbox{and}\quad
\{u\} =  \omega^e_+ \: u_+  + \omega^e_-\:  u_-,
\end{equation}
where $u_+$ and $u_-$ are the traces of $u_{|\tau_+}$ and $u_{|\tau_-}$ on $e$, respectively, while
$n_+$ and $n_-$ are the unit outer normal to $\partial \tau_+$ and $\partial\tau_-$, respectively.
And, for each edge $e\subset\mathcal{E}_h^\partial$, with $e$ being an edge of some $\tau$,
%which is an edge of $\tau \in {\cal T}_h$
we define
\begin{eqnarray}
[u]=u\: n  \quad \mbox{and} \quad \{u\}= u,
\end{eqnarray}
where $n$ is the unit outer normal to $e\subset \partial \tau$, and $u$ is the trace of $u_{|\tau}$ onto $e$.

We define the $L^2$ inner products over the elements and the edges respectively as follows,  
$$
 (u,v)_{{\cal T}_h}=\sum_{\tau\in {\cal T}_h} (u,v)_\tau \quad \mbox{and} \quad 
 (u,v)_{\mathcal{E}_h}=\sum_{e\in \mathcal{E}_h} (u,v)_e
$$
for $u,v \in L^2(\Omega)$. % which are smooth enough.

We consider a family of standard piecewise linear finite element spaces $V^h\subset L^2(\Omega)$, built on a family of triangulations $\{{\cal T}_h\}$, given as
$$
 V^h=\{u\in L^2(\Omega): \forall \tau \in {\cal T}_h \:u_{|\tau}\in P_1(\tau)\},
$$
where $P_1(\tau)$ is the space of linear polynomials defined over the element $\tau$.
Since, for a linear function, its gradient being constant over $\tau$, we have that $\int_\tau \alpha \nabla u \cdot \nabla v \: d x= (\nabla u \cdot \nabla v) \int_\tau \alpha \: d x$. Thus we see that our  assumption on $\alpha$ is natural here.

We can now introduce a family of discrete problems, based on the symmetric interior penalty Galerkin (SIPG) method as follows. Find $u_h^*\in V^h$ 
\begin{equation}\label{eq:IP_dp}
   a(u_h^*,v)=f(v) \qquad \forall v \in V^h,
\end{equation}
where 
$$
 a(u,v)=(\alpha \nabla u,\nabla v)_{{\cal T}_h} - (\{\alpha \nabla u\}[v])_{\mathcal{E}_h} -(\{\alpha \nabla v\}[u])_{\mathcal{E}_h} + \gamma  (S_h[u],[v])_{\mathcal{E}_h},
$$
for all $u,v \in V^h$,  $S_h$ is a piecewise constant function over the edges of ${\cal E}_h$,
and $\gamma=\mbox{constant}>0$ is a penalty term, cf. \cite{Z.Cai:2011:DGF}. $S_h$ when restricted to $e\in {\cal E}_h^0$,  is defined as follows, cf. \cite{Dryja:2003:ODG,Burman:2006:DDM,Z.Cai:2011:DGF},
$$
  S_{h|e}=h_e^{-1} (\omega^e_+\alpha_+ + \omega^e_-\alpha_-)=
  h_e^{-1}\frac{2}{\frac{1}{\alpha_+}+\frac{1}{\alpha_-}} \quad \mbox{on} \quad \overline{e}=\partial\tau_+\cap\partial\tau_-,
$$
with $h_e$ being the length of the edge $e \in \mathcal{E}_h$. 
With the harmonic average satisfying
\begin{equation}
 \alpha_{min} \leq \frac{2}{\frac{1}{\alpha_+}+\frac{1}{\alpha_-}} \leq 2 \alpha_{min} \quad\mbox{where}\quad \alpha_{min}=\min(\alpha_+,\alpha_-),
\end{equation}
we get
\begin{equation}\label{eq:est-Sh}
   h_e^{-1}\alpha_{min} \lesssim S_{h|e}\lesssim   h_e^{-1}\alpha_{min}.
\end{equation}
For $e\in {\cal E}_h^{\partial}$ we set 
$$
 %\omega^e=1 \qquad 
 S_{h|e}=h_e^{-1}\alpha_{|\tau}.
$$
If penalty parameters is large enough than the discrete problem has a unique solution, cf. Lemma~2.3 in \cite{Z.Cai:2011:DGF}.

We use the standard local nodal basis on each triangle $\tau\in {\cal T}_h$ to represent a function $u \in V^h$. Hence, $u\in V^h$ restricted to a triangle $\tau\in {\cal T}_h$ can be represented as $u_{|\tau}=\sum_{x\in \mathcal{V}_h(\tau)} u(x)\phi_x^\tau$
where $\mathcal{V}_h(\tau)$ is the set of vertices of $\tau$, and $\phi_x^\tau$ is a linear basis function defined over $\tau$ such that it takes the value one at the vertex $x$ and zero at the other two vertices of $\tau$. $\phi_x^\tau$ extends to the rest of the domain as zero. Consequently, any function $u\in V^h$ is represented as 
$$
 u =\sum_{\tau \in {\cal T}_h} \sum_{x\in \mathcal{V}(\tau)} u(x)\phi_x^\tau.
$$
In this basis, the degrees of freedom (dofs) are the function values associated with the vertices in the set
%$\overline{\Omega}_h := \bigcup_{\tau\in T_h} \mathcal{V}(\tau)$.
$\overline{\Omega}_h := \{ \mathcal{V}_h(\tau): \tau\in {\cal T}_h(\Omega) \}$. We call this set the discontinuous Galerkin or the DG vertex or node set. Vertices (or nodes) occupying the same geometric space may have different function values associated with them if the triangles they belong to are different, and therefore appear in the set $\overline{\Omega}_h$ with indices of their respective triangles. So, each geometric vertex $x$ may correspond to one or several DG vertices (nodes). We use $\nu(x)$ to denote the set of all DG vertices (nodes) associated with $x$.

%We could also use another local basis, e.g. Crouzeix-Raviart type of basis related to the midpoints of the edges of fine triangles.

Below, we introduce the bilinear form $\hat a(\cdot,\cdot)$, and restate its equivalence to the bilinear form $a(\cdot,\cdot)$ in Lemma \ref{lem:equivBF} (cf. Lemma~2.3 in \cite{Z.Cai:2011:DGF}).
For $u,v\in V^h$ let
\begin{eqnarray}
 \hat{a}(u,v)&=&(\alpha \nabla u,\nabla v)_{{\cal T}_h} + (S_h[u],[v])_{\mathcal{E}_h}
   \qquad.
 %\\  \hat{a}_0(u,v)&=&(\alpha \nabla u,\nabla v)_h + \sum_{e\in\mathcal{E}_h} (S_h P_e [u],[v])_e
% \\   a_0(u,v)&=&(\alpha \nabla u,\nabla v)_h - (\{\alpha \nabla u\}[v])_{\mathcal{E}_h} -(\{\alpha \nabla v\}[u])_{\mathcal{E}_h} + \sum_{e\in\mathcal{E}_h} (S_h P_e [u],[v])_e,
\end{eqnarray}
%where $P_eu=\frac{1}{|e|}\int_e u(s) \: d s$ is the $L^2$ orthogonal projection onto the space of constants over a fine edge $e$.

\begin{lemma}\label{lem:equivBF}
The norms induced by the bilinear forms $a(\cdot,\cdot)$ and $\hat{a}(\cdot,\cdot)$ 
are equivalent with constants independent of the mesh parameter $h$ or the coefficient $\alpha$ 
if the penalty term is larger then a positive constant which is dependent only on the geometry of
all triangles in the triangulation and is independent of $h$ and the contrast $\alpha$.
\end{lemma}
\begin{remark}
The upper bound of Lemma~\ref{lem:equivBF} is relatively easy to show; the difficult part is to show the lower bound, that is the bilinear form $a(\cdot,\cdot)$ is coercive in the norm induced by the bilinear form $\hat{a}(\cdot,\cdot)$. In Section~2.4 of \cite{Z.Cai:2011:DGF}), this coercivity constant is given in a very precise way; however the formula is a bit technical, and therefore we do not present it here.
\end{remark}

\section{Additive Schwarz method}\label{sec:Prec}
Let $\overline{\Omega}=\bigcup_{k=1}^N \overline{\Omega}_k$ be the non-overlapping decomposition of $\Omega$ into disjoint open substructures $\Omega_k, \ k=1,\ldots,N$, each being a sum of fine triangles $\tau\in{\cal T}_h$ and edges $e_h\in{\cal E}_h$.

Let $\Gamma_{kl}$ be the open interface between $\Omega_k$ and $\Omega_l$, with $\overline{\Gamma}_{kl} = \partial\Omega_k\cap\partial\Omega_l$ for $k\not=l$ being its closed interface. The geometric point or vertex where the interfaces meet are called the crosspoint and is typically denoted by $c_r$. Note that $\nu(c_r)$ denotes the set of DG vertices or nodes associated with the crosspoint $c_r$.  
Now, define $\Gamma$, the global interface, as the sum of all closed interfaces.

For each subdomain interface $\Gamma_{kl}$, we define a patch, denoted by $\Gamma_{kl}^{\delta}$, as the sum of all triangles having at least a vertex on $\Gamma_{kl}$. Note that $\Gamma_{kl}$ is $\overline{\Gamma}_{kl}$ minus the crosspoints. For the ease of explanation, we assume that the patches are disjoint in the sense that they may share a vertex (a crosspoint) geometrically, but not a triangle; this is however not necessary in practice, cf. Remark~\ref{remark:overlappatch}. Each patch $\Gamma_{kl}^{\delta}$ has two disjoint subpatches, $\Gamma_{kl}^{{\delta},i}=\Gamma_{kl}^{\delta}\cap\Omega_i$ for $i=k,l$, giving
$
 \Gamma_{kl}^{\delta}=\Gamma_{kl}^{{\delta},k}\cup\Gamma_{kl}^{{\delta},l}.
$
The sum of all subpatches belonging to a subdomain $\Omega_k$, called the discrete boundary layer of $\Omega_k$ and denoted by $\Omega_k^{\delta}$, is defined as
$$
 \Omega_k^{\delta}=\bigcup_{\Gamma_{k l} \subset \partial\Omega_k \cap \Gamma} 
       \Gamma_{k l}^{{\delta},k}.     
$$ 

\begin{figure}[!ht]
 \centerline{
 \SetLabels
 \L\B (.15*.67)  $\Omega_4$ \\
 \L\B (.65*.67)  $\Omega_3$ \\
 \L\B (.15*.10)  $\Omega_1$ \\
 \L\B (.65*.10)  $\Omega_2$ \\
 \L\B (.39*.62)  $\Gamma_{34}^{\delta}$ \\
 \L\B (.62*.39)  $\Gamma_{23}^{\delta}$ \\
 \L\B (.39*.15)  $\Gamma_{12}^{\delta}$ \\
 \L\B (.17*.39)  $\Gamma_{41}^{\delta}$ \\
 \endSetLabels
 %\ShowGrid
 \AffixLabels{
 \includegraphics[height=5.5cm, width=6cm]{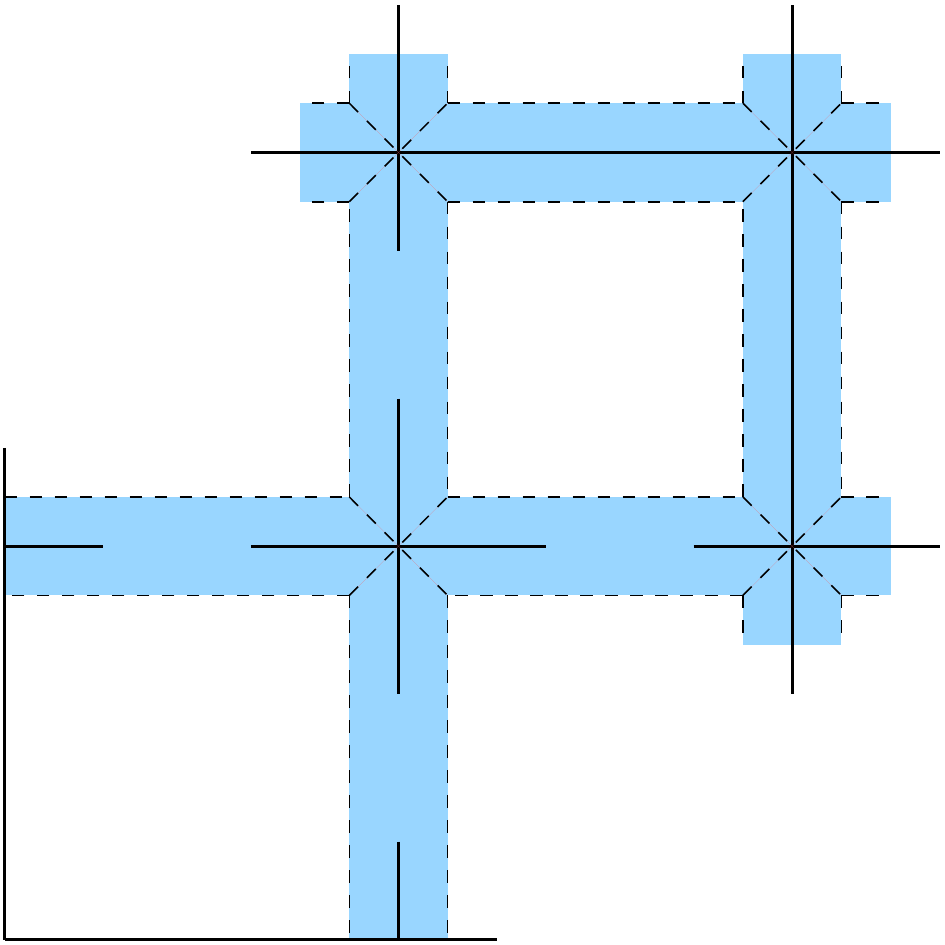}}
 }
 \caption{Illustrating disjoint patches (shaded regions) for both interior and boundary subdomain. Each patch covers an interface between two neighboring subdomains, extending to the boundary, as shown.}\label{fig:patches}
\end{figure}
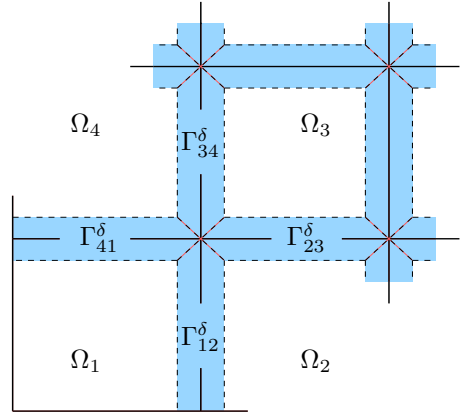

Let $V_k$ be the subspace of $V^h$ of functions that are zero over all the fine triangles that are in  $\overline{\Omega}_l$ where $l\neq k$, and let $V_{k,0}$ be the subspace of $V_k$ of functions that are zero on the boundary layer $\Omega_k^{\delta}$,
that is
%$$
% V_{k,0}=\left\{u\in V_k: u(x)=0 \quad x\in \overline{\tau}, \; \tau \in \Omega_k^h \cup \bigcup_{l\not =k} \overline{\Omega}_l\right\}
%$$
$$
 V_{k,0}=\left\{u\in V_k: \ u(x)=0 \ \ \forall x\in {\tau}, \ \tau \subset \Omega_k^{\delta} \right\}.
$$
The bilinear form ${a}(\cdot,\cdot)$ is positive definite over both $V_k$ and $V_{k,0}$. We now introduce an orthogonal  projection $\mathcal{P}_k:V^h\rightarrow V_{k,0}$ as follows,
\begin{equation} \label{eq:dh-pr}
 {a}(\mathcal{P}_ku,v)= {a}(u,v) \quad \forall v \in V_{k,0}.
\end{equation}

Let  $\mathcal{P} =\sum_{k=1}^N \mathcal{P}_k$.
Since the supports of functions in $V_{k,0}$ and $V_{l,0}$ are disjoint for $k\not =l$ and functions from these spaces are zero on edges contained  in $\Gamma$, the images of $\mathcal{P}_ku$ and $\mathcal{P}_l u$ are orthogonal in any of the bilinear forms $a(\cdot,\cdot)$ and $\hat{a}(\cdot,\cdot)$.
%and hence the operator $\mathcal{P} $ is an ${a}$-orthogonal ($\hat{a}$-orthogonal) projection.

Further, let $\mathcal{H} =I-\mathcal{P} : V^h\rightarrow V^h$ be the corresponding discrete harmonic extension operator. This operator has the minimization property stating that
\begin{equation}\label{eq:min_prop_dhe}
 {a}(\mathcal{H} u,\mathcal{H} u)=\min_{u\in V^h}\left\{a(u,u): \; u(x)=\mathcal{H} u(x) \; \; \forall x\in \mathcal{V}_h(\tau), \; \tau \subset \bigcup_k \Omega_k^{\delta}\right\}.
\end{equation}
%This property holds also locally, i.e. if we consider $\mathcal{H}_k u$  the restriction of $\mathcal{H} u$ to $\overline{\Omega}_k$ then
%$$ \hat{a}(\mathcal{H}_ku,\mathcal{H}_k u)=\min\{u\in V_k: u(x)=\mathcal{H}_k(x), \; x\in \mathcal{V}(\tau),\; \tau \in  \Omega_k^h\}$$

\subsection{Local spaces}
For the additive Schwarz decomposition, let $V_k$ be the local subspace associated with the subdomain $\Omega_k$, giving that 
$$ 
  V^h=\sum_{k=1}^N V_k \quad\mbox{with}\quad V_k\cap V_l =\{0\} \ \ \mbox{for} \ k\not=l.
$$
%with $V_k\cap V_l =\{0\}$, for $k\not=l$. 
We note here that, even when two neighboring subdomains share an edge, their subspaces may not necessarily be orthogonal to each other in the inner product induced by the two bilinear forms $a(\cdot,\cdot)$ and ${\hat a}(\cdot,\cdot)$. This is because of the term
\begin{equation}
%  \hat{a}(u,v)=\sum_{e \subset \Gamma_{k l}} \int_e S_h [u][v]\:d s \qquad \forall u\in V_k, v\in V_l,
  \sum_{e \subset \Gamma_{k l}} \int_e S_h [u][v]\:d s \qquad \forall u\in V_k, v\in V_l,
\end{equation}
being present in the inner product on both subspaces when $\Gamma_{kl}$ is nonempty. In this sense, the method can be considered as an overlapping Schwarz method with the minimal overlap.

\subsection{Coarse space}
In this section, we introduce our coarse space which consists of two components, a spectral component, and a non-spectral component. The way the components are to be built play an important role in making our method robust and effective. 

Let $V^h(\Gamma_{k l}^{\delta})=\left\{u\in V^h: \; u_{|\tau}=0, \; \tau\not\subset \Gamma_{k l}^{\delta}\right\}$ be the space of functions that are equal to zero on all elements which do not belong to the patch $\Gamma_{k l}^{\delta}$.
\begin{remark}\label{remark:overlappatch}
 In a more general case when two patches  $\Gamma_{k l}^\delta$ and $\Gamma_{k j}^\delta$ sharing a crosspoint $c_r$ are not disjoint, i.e. there exists a fine triangle $\tau$ with an edge $e_1$ on $\Gamma_{k l}$ and another edge $e_2$ on $\Gamma_{k j}$, we need to slightly modify the definition of the two spaces $V^h(\Gamma_{k l}^\delta)$ and $V^h(\Gamma_{k j}^\delta)$. We simply add an extra condition to their functions, namely by saying that for any function $u\in V^h(\Gamma_{k l}^\delta)$, $u$ takes zero on $\partial \tau \cap \Gamma_{k j}=\overline{e}_2$, and for any function $u\in V^h(\Gamma_{k j}^\delta)$, $u$ takes zero on $\partial \tau \cap \Gamma_{k l}=\overline{e}_1$.
\end{remark}
On $V^h(\Gamma_{k l}^{\delta})$, we introduce two symmetric bilinear forms $a_{k l}(u,v)$ and $b_{k l}(u,v)$  as follows: 
\begin{eqnarray}
 a_{k l}(u,v)&:=& \sum_{\tau \subset \Gamma_{ k l}^{\delta}} \int_\tau\alpha \nabla u \cdot \nabla v \: d x + 
 \sum_{e \subset \Gamma_{k l}^{\delta}\cup(\partial\Omega\cap\partial{\Gamma}_{k l}^{\delta})}
  S_h\int_e [u][v] \:d s, \\
  b_{k l}(u,v)&:= &h^{-2} (u,v)_{L_\alpha^2(\Gamma_{k l}^\delta)}=h^{-2}\int_{\Gamma_{k l}^{\delta}}\alpha\: u \: v \: d x.
\end{eqnarray}
The second sum in the bilinear form $a_{k l}(\cdot,\cdot)$ is taken over all edges of the triangles $\tau \subset \Gamma_{k l}^{\delta}$, which are either in the interior of the patch $\Gamma_{k l}^{\delta}$ or on the boundary $\partial\Omega\cap\partial{\Gamma}_{k l}^{\delta}$.
Edges that lie on the boundary of the patch and are in the interior of a subdomain at the same time are not in this sum.

Note that for $u,v\in V^h(\Gamma_{k l}^{\delta})$
$$
a_{k l}(u,v) = \hat{a}(u,v) - \sum_{e \subset \partial\Gamma_{k l}^\delta\setminus\partial\Omega}
  S_h\int_e [u][v] \:d s. 
$$
In cases where $\partial\Omega \cap \overline{\Gamma}_{k l}^{\delta}=\emptyset$ this form is only positive semi-definite over the space $V^h(\Gamma_{k l}^{\delta})$. Therefore we introduce $V_v(\Gamma_{k l}^{\delta})\subset V^h(\Gamma_{k l}^{\delta})$, the subspace of $V^h(\Gamma_{k l}^{\delta})$ of functions which are equal to zero at the nodes of $\nu(c_r)$, where $c_r$ are the crosspoints which are typically the endpoints of $\Gamma_{k l}$. For all functions in this subspace we can make the following proposition
\begin{proposition}\label{prop:aklpatchformSPD}
The bilinear form $a_{k l}(u,v)$, where $u,v \in V_v(\Gamma_{k l}^{\delta})$, is symmetric and positive definite.
\end{proposition}
\begin{proof}
Clearly the bilinear form is symmetric. To prove that it is positive definite, we only need to show that $a_{k l}(u,u)=0$ if and only if $u=0$. Let $u\in V_v(\Gamma_{k l}^{\delta})$, and $a_{k l}(u,u)$ be equal to zero, hence $\nabla u$ is zero over each triangle $\tau \subset \Gamma_{k l}^{\delta}$ which means that $u$ is piecewise constant over $\Gamma_{k l}^{\delta}$. We also have, for each interior edge $e$, $\int_e [u]^2\: d s =0$. Hence, $u$ is constant over the patch $\Gamma_{k l}^{\delta}$, consequently, $u$ is zero at the vertices of $\nu(c_r)$ for each crosspoint $c_r$ of $\Gamma_{k l}$. By the definition of $\Gamma_{k l}$ it contains at least one crosspoint. This yields that $u=0$ over the patch. Since $V_v(\Gamma_{k l}^{\delta})$ is finite dimensional we get that the form is positive definite over this space.  
\end{proof}

We first define the non-spectral component of the coarse space. It is constructed in the similar fashion as the standard multiscale finite element space is constructed. We call this component a multiscale component, and denote it by $V_{ms}$. The functions $u\in V_{ms}$ are determined by its values at the DG vertices $\nu(c_r)$ of the crosspoints $c_r$. $V_{ms}\subset V^h$ is the space of functions $u$ which satisfy the condition that
\begin{equation}\label{eq:def-Vms}
 a_{k l}(u,v)=0 \quad \forall v \in V_v(\Gamma_{k l}^{\delta}),
\end{equation} 
over each patch $\Gamma_{k l}^{\delta}$, which is guaranteed by Proposition~\ref{prop:aklpatchformSPD}, and that they are discrete harmonic in the sense that $u=\mathcal{H}u$ as defined in Section \ref{sec:Prec}.

The spectral component of the coarse space is based on solving a generalized eigenvalue problem locally on each patch, which is defined as follows. On each patch $\Gamma_{k l}^{\delta}$, find all pairs $(\lambda_j^{k l},\psi_j^{k l})
\in R_+\times V_v(\Gamma_{k l}^{\delta})$ such that
\begin{equation}
  a_{k l}(\psi_j^{k l},v)=\lambda_j^{k l} b_{k l}(\psi_j^{k l},v)
  \qquad \forall v \in V_v(\Gamma_{k l}^{\delta})
\end{equation}
and $\|\psi_j^{k l}\|_{b_{kl}}=1$, with $\|\psi_j^{k l}\|_{b_{kl}}$ as the norm induced by the bilinear form $b_{k l}(u,v)$.

We assume that the eigenvalues are indexed in the increasing order, with $\lambda$ being a number between the $M$-th and the ($M+1$)-th lowest eigenvalues, as
\begin{equation} \label{eq.threshold}
  0 < \lambda_1^{k l}\leq \lambda_2^{k l}\leq \ldots \leq \lambda_M^{kl} \leq \lambda \leq \lambda_{M+1}^{kl} \leq \ldots \leq \lambda_{N_{k l}}^{k l}
\end{equation}
where $N_{k l}=dim(V_v(\Gamma_{k l}^{\delta}))$.
We define $\Pi_M^{k l}:V_v(\Gamma_{k l}^{\delta})\rightarrow V_v(\Gamma_{k l}^{\delta})$ as the $b_{k l}$-form orthogonal projection  onto the space 
$span(\psi_1^{k l},\ldots,\psi_M^{k l})$, as
\begin{equation}\label{eq:def_PI_kl}
  \Pi_M^{k l} u = \sum_{j=1}^M b_{k l}(u,\psi_j^{k l}) \psi_j^{k l}.
\end{equation}
with $0\leq M \leq dim(V_v(\Gamma_{k l}^{\delta}))$. The integer parameter
$M=M(\Gamma_{k l})$ is either preset or chosen automatically by setting a threshold for the eigenvalues. Our estimates below  will depend on the choice of $M$ for the patches.
 
We are now ready to introduce the spectral component; it is the sum of patch subspaces $V_{k l}^{eig,M}, \ \Gamma_{kl}\subset\Gamma$, of $V^h$, defined as the following.
\begin{equation}
 V_{k l}^{eig,M}=span(\Psi_1^{k l},\ldots,\Psi_M^{k l}),
\end{equation}
where $\Psi_j^{k l}$ is the extension of $\psi_j^{k l}$, first as zero on the triangles that are on the boundary layers minus the patch $\Gamma_{k l}^{\delta}$, and then as discrete harmonic further inside the subdomains in the sense as described in Section \ref{sec:Prec}.
The functions of this space have a support which is the union of the patch $\Gamma_{k l}^{\delta}$ and the interior of both $\Omega_k$ and $\Omega_l$, as shown in Figure~\ref{fig:support}.

%\begin{figure}[!ht]
% \centering
% \includegraphics[height=6cm]{supportII.pdf}
% \caption{Here we see the support of $V_{k l}^{eig,M}$. That is the inner elements of $\Omega_k$ and $\Omega_l$ and the patch $\Gamma^h_{kl}$. Full lines are subdomain boundaries, patches are indicated by doted lines and cross-points are given by in red dots.}\label{fig:support}
%\end{figure} 

\begin{figure}[!ht]
 \centerline{
 \SetLabels
 \L\B (.25*.45)  $\Omega_k$ \\
 \L\B (.65*.45)  $\Omega_l$ \\
 \L\B (.49*.50)  $\Gamma_{kl}^{\delta}$ \\
 \endSetLabels
 %\ShowGrid
 \AffixLabels{
 \includegraphics[height=5cm, width=10cm]{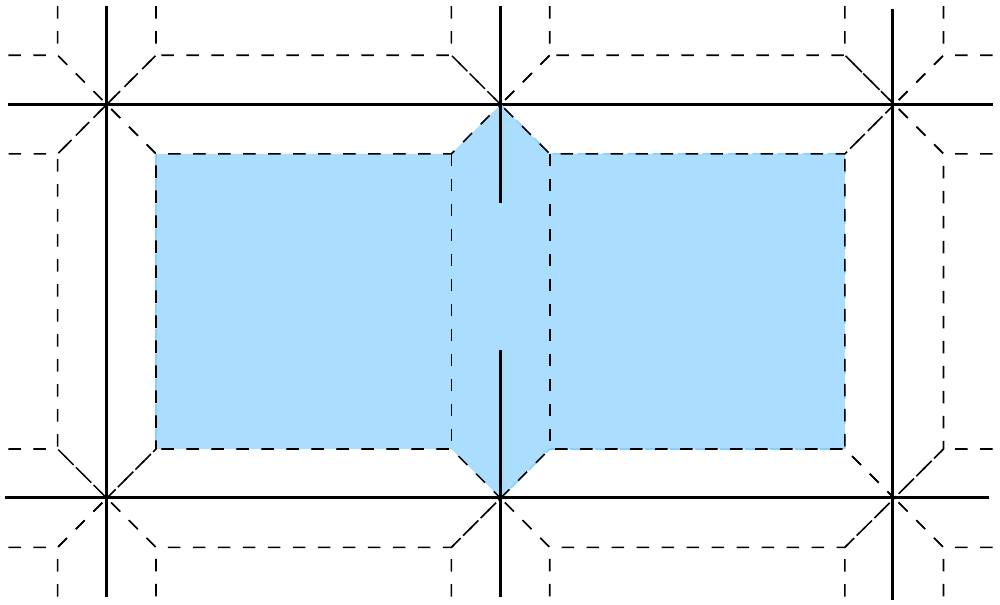}}
 }
 \caption{Illustrating the support of $V_{k l}^{eig,M}$, shown as shaded area which includes the patch $\Gamma^{\delta}_{kl}$ and the interiors of both $\Omega_k$ and $\Omega_l$.}\label{fig:support}
\end{figure} 

The coarse space is defined as the sum of $V_{ms}$, the non-spectral multiscale component, and $\{V_{k l}^{eig,M}\}$, the spectral component, as follows,
\begin{equation}
V_0= V_{ms}+\sum_{\Gamma_{k l} \subset \Gamma}  V_{k l}^{eig,M},
\end{equation}
in other words the coarse space is a multiscale like coarse space enriched with patch spectral subspaces.

\subsection{Preconditioned system}
We define the coarse and the local projection like operators $\{T_k\}_{k=0}^N$ as $T_k:V^h\rightarrow V_k$, $k=0,1,\ldots, N,$ satisfying  
\begin{equation} \label{eq:DDProj}
 {a}(T_k u,v)=a(u,v) \qquad \forall v \in V_k,
\end{equation}
%$T_0$ is the coarse ASM projection like operator.
and the corresponding additive Schwarz operator $T$ as 
\begin{equation}
 T=T_0 + \sum_{k=1}^N T_k.
\end{equation}
Now, following the Schwarz framework, cf. e.g. \cite{Toselli:2005:DDM}, the discrete formulation (\ref{eq:IP_dp}) can be written equivalently as
\begin{equation}
 T u^*_h =g,
\end{equation}
which is a preconditioned version of the original system, where $g=\sum_{k=0}^Ng_k$ and $g_k=T_ku^*_h$.

\begin{remark}\label{rem:inex-b-f}
 If we replace the exact bilinear form $a(u,v)$ by the inexact bilinear form $\hat{a}(u,v)$, on the left of the equality in the definition of $T_k$, cf. (\ref{eq:DDProj}), we will get a second variant of the preconditioner with inexact solvers for the sub problems, but a similar convergence estimate as the exact version.
\end{remark}

\subsection{Condition number}
We get the following condition number bound for the preconditioned system.
\begin{theorem}\label{thm:Cond_est}
Let $M=M(\Gamma_{kl})$ be the number of enrichment used on each subdomain interface $\Gamma_{kl}$ between two neighboring subdomains $\Omega_k$ and $\Omega_l$. Then for any $u\in V^h$, the following is true for the additive Schwarz operator $T$, i.e.
\begin{equation}
 \left(1+\max_{\Gamma_{k l}\subset \Gamma}\frac{1}{\lambda^{k l}_{M+1}}\right)^{-1} a(u,u) \lesssim
 a(T u, u) \lesssim a(u,u),
\end{equation}
where $\lambda^{kl}_{M+1}$ is the $M+1$-st lowest eigenvalue of the generalized eigenvalue problem defined on the thin patch associated with the interface $\Gamma_{kl}$.
\end{theorem}
The proof of this theorem is given in the next section.

\section{Proof of Theorem~\ref{thm:Cond_est}}\label{sec:Proofs}

The proof is based on the abstract Schwarz framework, see \cite{Smith:1996:DDP,Toselli:2005:DDM,Mathew:2008:DDM} for more details.
Accordingly, there are three key assumptions that need to be verified; these are the assumptions on the stability of the decomposition, the strengthened Cauchy-Schwarz inequality between the local subspaces, and the local stability of the inexact bilinear forms if any. The second assumption is verified using a simple coloring argument. The third assumption needs to be verified if $\hat{a}(\cdot,\cdot)$ is used instead of the exact bilinear form $a(\cdot,\cdot)$, cf. Remark~\ref{rem:inex-b-f}, in which case it is a simple consequence of the equivalence between the two as given in Lemma~\ref{lem:equivBF}.
We are then left with the assumption on the stability of the decomposition, which needs to be verified, and which is shown in Lemma~\ref{lem:stable_dec} below.

We need a few technical tools. The first one is the following set of local inverse inequalities, cf. Lemma \ref{lem:IPinvineq}.
\begin{lemma}\label{lem:IPinvineq}
Let $u$ be a function such that $u\in  V^h$. Then for any $\tau \in {\cal T}_h$ or $e \in \mathcal{E}_h$ we have
\begin{eqnarray*}
 \int_{\tau} \alpha |\nabla u |^2\: d x &\lesssim & h^{-2} \int_{\tau} \alpha | u |^2 \: d x,\\
 \int_e  S_h [u]^2 \: d s &\lesssim& h^{-2} \int_{\tau_+} \alpha_+ | u |^2 \: d x+h^{-2} \int_{\tau_-} \alpha_- | u |^2 \: d x,   \quad e\in \mathcal{E}_h^0,\\
  \int_e  S_h [u]^2 \: d s &\lesssim& h^{-2} \int_{\tau} \alpha | u |^2 \: d x,  \quad e\in \mathcal{E}_h^\partial.
\end{eqnarray*}
where $e \in \mathcal{E}_h^0$ in the second inequality, is the edge shared by the elements $\tau_+$ and $\tau_-\in {\cal T}_h$.
\end{lemma}
\begin{proof}
The first inequality is the classical local inverse inequality, cf. e.g. \cite{Brenner:2008:MTF}.
The last two inequalities follow from the trace theorem over $e$, a scaling argument, and the fact that 
$h_eS_{h|_e}\leq 2 \min\{\alpha_+,\alpha_-\}$ for $e \in \mathcal{E}_h^0$, cf. (\ref{eq:est-Sh}).
\end{proof}

The next technical tool is a corollary of Lemma~\ref{lem:IPinvineq}.
\begin{corollary}\label{cor:disharmest}
Let $u=\mathcal{H} u\in V^h$ be a discrete harmonic function as defined in Section \ref{sec:Prec}. Then
\begin{equation} \label{eq:CorDH}
{a}(u,u)\lesssim 
\sum_{k=1}^N h^{-2}\|\alpha^{1/2} u\|_{L^2(\Omega_k^{\delta})}^2.
\end{equation}
\end{corollary}
\begin{proof}
Let $\hat{u}\in V^h$ be equal to $u$ on all the boundary layers, that is the vertices of $\Omega_k^{\delta}$ for $k=1,\ldots,N$, and be extended by zero further inside the subdomains, out of the boundary layers.
 
By the minimizing property of the discrete harmonic extension, cf. (\ref{eq:min_prop_dhe}), and Lemma~\ref{lem:equivBF}, we get
\begin{eqnarray*}
 |u|^2_a&\leq& |\hat{u}|^2_a \lesssim |\hat{u}|^2_{\hat{a}} =\sum_{k=1}^N \left( \sum_{\tau \subset \Omega_k^{\delta}} 
  \|\alpha^{1/2}\nabla \hat{u}\|_{L^2(\tau)}^2 + \sum_{e\subset \overline{\Omega}_k^{\delta}}\int_e S_h[\hat{u}]^2\: ds \right) .
\end{eqnarray*}
By the first inequality of Lemma~\ref{lem:IPinvineq}, the first sum above can be estimated by the square of  $L^2$ norm of $\hat{u}$ scaled by $h^{-2}$  over all boundary layers. Thus, since $\hat{u}=u$ on the boundary layers, we get
\begin{equation}\label{eq:proof_cor_inv}
 \sum_{\tau\subset \Omega_k^{\delta}} \|\alpha^{1/2}\nabla \hat{u}\|_{L^2(\tau)}^2 \lesssim 
 \sum_{\tau\subset \Omega_k^{\delta}} 
  h^{-2}\|\alpha^{1/2}\hat{u} \|_{L^2(\tau)}^2=h^{-2}\|\alpha^{1/2}{u} \|_{L^2(\Omega_k^{\delta})}^2.
\end{equation}
%\begin{equation}\label{eq:proof_cor_inv}
% \sum_{k=1}^N \sum_{\tau\in \Omega_k^h} \|\alpha^{1/2}\nabla \hat{u}\|_{L^2(\tau)}^2 \lesssim 
% \sum_{k=1}^N \sum_{\tau\in \Omega_k^h} 
%  h^{-2}\|\alpha^{1/2}\hat{u} \|_{L^2(\tau)}^2=\sum_{k=1}^Nh^{-2}\|\alpha^{1/2}{u} \|_{L^2(\Omega_k^h)}^2.
%\end{equation}
For the edge term, $\sum_{e\subset \overline{\Omega}_k^{\delta}}\int_e S_h [\hat{u}]^2\: d s$, we estimate it by estimating its edge integral term separately for each case of $e$ in the sum. 

Let $e\subset\overline{\Omega}_k^{\delta}$ be an edge of some $\tau\subset\Omega_k^{\delta}$, and be lying on $\partial{\Omega}_k^{\delta}\cap \partial \Omega$, then by the third inequality of Lemma~\ref{lem:IPinvineq}, we estimate the edge integral as
$$ 
 \int_e S_h [\hat{u}]^2\: d s \lesssim h^{-2}\int_\tau\alpha \hat{u}^2\: d x.
$$
%which is a scaled $L^2$ norm over an element in the boundary layer $\Omega_k^h$. 
Let $e\subset\overline{\Omega}_k^{\delta}$ be the edge common to two triangles $\tau_+$ and $\tau_-$, then by
the second  inequality of Lemma~\ref{lem:IPinvineq}, we get
$$ 
 \int_eS_h [\hat{u}]^2\: d s \lesssim h^{-2}\int_{\tau_+\cup\tau_-}\alpha \hat{u}^2\: d x.
$$
There are three cases to consider in this case. In the first, both $\tau_+\subset \Omega_k^{\delta}$ and $\tau_-\subset \Omega_k^{\delta}$. The edge integral is then bounded by the sum of the squares of the $L^2$ norm scaled by $h^{-2}$ over the two triangles. In the second, $e\subset\Gamma_{k l}$, in which case one of the two triangles, $\tau_+$ and $\tau_-$, is in the boundary layer $\Omega_k^{\delta}$ and the other one in the boundary layer $\Omega_l^{\delta}$. The edge integral is then bounded by the sum of the squares of the $L^2$ norm scaled by $h^{-2}$ over the two triangles. 
%$\sum_{\Gamma_{k l}\subset \Gamma\cap\partial\Omega_k}h^{-2} \|\alpha^{1/2}u\|_{L^2(\Gamma_{k l}^h)}^2$.
Finally, the case when $e \subset \partial\Omega_k^{\delta}\setminus\partial\Omega_k$, in which case we have
$\tau_+\subset \Omega_k^{\delta}$ and $\tau_-\subset \Omega_k\setminus\Omega_k^{\delta}$ with $\hat{u}=0$ on $\tau_-$. We estimate the edge integral as
$$ 
 \int_eS_h [u]^2\: d s \lesssim h^{-2}\int_{\tau_+\cup\tau_-}\alpha \hat{u}^2\: d x=
 h^{-2}\int_{\tau_+}\alpha \hat{u}^2\: d x, % =h^{-2}\int_{\tau_+}\alpha u^2\: d x,
$$ 
which is again the square of the  $L^2$ norm scaled by $h^{-2}$ over the triangle in the boundary layer $\Omega_k^{\delta}$.
Now adding all contributions from the edge integrals, and noting that $\hat{u}=u$ on the boundary layers, we can estimate the edge term as follows. 
\begin{equation} \label{eq:edgeterm}
 \sum_{e\subset \overline{\Omega}_k^{\delta}} \int_e S_h[\hat{u}]^2\: ds \lesssim
  h^{-2} \|\alpha^{1/2}{u}\|^2_{L^2(\Omega_k^{\delta})} + 
 \sum_{\Gamma_{k l} \subset \Gamma\cap \partial\Omega_k} h^{-2} \|\alpha^{1/2}{u}\|^2_{L^2(\Omega_l^{\delta})}.
\end{equation}
Further, adding (\ref{eq:edgeterm}) to (\ref{eq:proof_cor_inv}) and summing over the subdomains we get the right hand side of (\ref{eq:CorDH}), and the proof then follows.
\end{proof}

The following result can be obtained through a standard algebraic reasoning, see for instance \cite{Galvis:2010:DDMR,Spilane:2014:ARC,Loneland:2015:ANH} for similar applications.
\begin{lemma}\label{lem:PIM-est}
Let $\Pi_M^{k l}$ be as defined in (\ref{eq:def_PI_kl}), then
it is both the $b_{k l}$- and $a_{k l}$-orthogonal projection onto same subspace, and
\begin{eqnarray*}
 \|u-\Pi_M^{k l}u\|_{a_{k l}}\leq \|u\|_{a_{k l}}, \qquad 
 \|\Pi_M^{k l}u\|_{a_{k l}}\leq \|u\|_{a_{k l}},
\end{eqnarray*}
and
\begin{eqnarray*}
  \|u-\Pi_M^{k l}u\|_{b_{k l}}^2 &\leq& \frac{1}{\lambda_{M+1}^{kl}}\|u-\Pi_M^{k l}u\|_{a_{k l}}^2. 
\end{eqnarray*}
\end{lemma}

The next lemma states certain properties of the patch bilinear form $a_{k l}(u,v)$.
\begin{lemma} \label{lem:est-patchen-glenrg}
Let $u\in V^h$ and $u_{k l}=u_{|\overline{\Gamma}_{k l}^h}$  be its restriction to the 
patch $\overline{\Gamma}_{k l}^\delta$  extended by zero outside. Then
\begin{equation} \label{eq:est-patchen-glenrg}
 a_{k l}(u_{k l},u_{k l})\leq \sum_{\tau \subset \Omega_k\cup\Omega_l}\int_\tau \alpha |\nabla u|^2\: d x 
 + \sum_{e\subset \overline{\Omega}_k\cup \overline{\Omega}_l} \int_e S_h [u]^2\: d s.
\end{equation} 
\end{lemma}
\begin{proof}
We first note that, in the bilinear form $a_{k l}(u,u)$, the edge integrals are taken over the edges of $\Gamma_{kl}^\delta$, that are either in the interior of the patch $\Gamma_{kl}^\delta$ or on the outer boundary $\partial \Omega$. Thus the proof follows directly from the definition of the bilinear forms as all triangle terms 
and edge terms of the form $a_{k l}$ are contained in the right hand side of (\ref{eq:est-patchen-glenrg}). 
\end{proof}

For $u\in V^h$, we introduce the following interpolation operators, the multiscale interpolant and the coarse interpolant, respectively, as
\begin{equation}
 I_{m s}: \ V^h \rightarrow V_{m s} \quad\mbox{and}\quad I_0: \ V^h \rightarrow V_0.
\end{equation}
For $u\in V^h$, let $I_{ms} u\in V_{ms}$ be the function such that 
$$
 I_{ms}u(x)=u(x) \quad \forall x\in\mathcal{V}(c_r)\; \forall c_r\in \Gamma,
$$
that is a multiscale function interpolating $u$ at the DG vertices $\nu(c_r)$
of all cross-points $c_r$ in $\Gamma$.
This yields that $u-I_{ms}u$ is zero at the DG vertices of crosspoints and thus 
$u-I_{ms}u$ restricted to a patch $\Gamma_{k l}^{\delta}$ is in the $V_v(\Gamma_{k l}^{\delta})$.

Next we define $I_0 u\in V_0$ on each patch $\Gamma_{k l}^h$ as 
\begin{equation}
 I_0 u =I_{ms} u+
  \Pi_m^{k l} (u - I_{ms} u ) \quad \mbox{on}\quad \overline{\Gamma_{k l}^{\delta}},
\end{equation}
with $\Pi_m^{k l}$ as defined in (\ref{eq:def_PI_kl}).
$I_0 u $ is extended inside as discrete harmonic.
We have the following bound for the coarse interpolant.
\begin{lemma}\label{lem:coarse-int-est}
For the coarse interpolant $I_0$ it holds that
\begin{equation}
 {a}(I_0u,I_0u) \lesssim \left(1+\max_{\Gamma_{k l}\subset \Gamma}\frac{1}{\lambda_{M+1}^{k l}}\right) a(u,u)  \quad\forall u\in V^h.
 \end{equation}
\end{lemma}
\begin{proof}
By a triangle inequality we get immediately that
$$
 \|I_0u\|_{{a}}\leq \|I_0u-u\|_{{a}}+\|u\|_{{a}}.
$$
It is suffices therefore to prove the bound for the first term on the right hand side of this inequality.
Define $w=u-I_0 u$ for $u\in V^h$. Note that $\mathcal{P}  w=\mathcal{P}  u$ and $\mathcal{H} w=\mathcal{H} u - \mathcal{H} I_0u=\mathcal{H} u-I_0 u$ as $I_0 u$ is discrete harmonic in the way described in Section \ref{sec:Prec}. Hence
$$
 \|I_0u-u\|_{a}= \|\mathcal{H} w\|_{a}+\|\mathcal{P}  w\|_{a}
  \leq \|\mathcal{H} u -I_0 u\|_{a}+\|\mathcal{P}  u\|_{a} \leq 
  \|\mathcal{H} u -I_0 u\|_{a}+\|u\|_{a}.
$$
We have used the fact that $\mathcal{P} $ is the orthogonal projection with respect to $a(\cdot,\cdot)$.

Next, we estimate $\|\mathcal{H} w\|_{a}=\|\mathcal{H} u -I_0 u\|_{a}$.
By Lemma~\ref{lem:equivBF} and Corollary~\ref{cor:disharmest} we have
\begin{eqnarray*}
 \|\mathcal{H} u -I_0 u\|_a&\lesssim& \|\mathcal{H} u -I_0 u\|_{\hat{a}}
 \lesssim \sum_{k=1}^N h^{-2} \|\alpha^{1/2}(\mathcal{H} u-I_0 u)\|_{L^2(\Omega_k^{\delta})}^2\\
 &\lesssim&
 \sum_{\Gamma_{k l}\subset \Gamma} h^{-2}\|\alpha^{1/2}(u-I_0 u)\|_{L^2(\Gamma_{k l}^{\delta})}^2
\end{eqnarray*}
Note that $u-I_0u=(I-\Pi_M^{k l})(u-I_{ms}u)\in V_v(\Gamma_{k l}^{\delta})$.
Hence by Lemma~\ref{lem:PIM-est} we get
\begin{eqnarray*}
 h^{-2}\|\alpha^{1/2}(u-I_0 u)\|_{L^2(\Gamma_{k l}^h)}^2& =&
 \|(I-\Pi_M^{k l})(u-I_{m s}u)\|_{b_{k l}}^2 \\
 &\leq& \frac{1}{\lambda_{M+1}^{k l}} 
 \|(I-\Pi_M^{k l})(u-I_{m s}u)\|_{a_{k l}}^2 \\ 
 &\leq& \frac{1}{\lambda_{M+1}^{k l}}  \|(u-I_{m s}u)\|_{a_{k l}}^2 \\
 &\leq&  \frac{1}{\lambda_{M+1}^{k l}} \|u\|_{a_{k l}}^2.
\end{eqnarray*}
 The last inequality follows from  the fact that $u-I_{m s} u$ restricted to 
 the patch $\Gamma_{k l}^{\delta}$ is $a_{k l}$-orthogonal to the space $V_v(\Gamma_{k l}^{\delta})$, what follows from the definitions of $V_{ms}$ and $I_{ms}$, cf. (\ref{eq:def-Vms}).
 
 Now, utilizing Lemma~\ref{lem:est-patchen-glenrg} for each patch, summing over the interfaces, and finally using Lemma~\ref{lem:equivBF}, the proof then follows.
\end{proof}

\begin{lemma}\label{lem:local-est}
Let $w_k$ be the restriction of $u-I_0u$ to $\overline{\Omega}_k$, and extended by zero to the other subdomains, then 
\begin{equation}
 \sum_{k=1}^N a(w_k,w_k) 
 \lesssim \left(1+\max_{\Gamma_{k l}\subset \Gamma\cap \partial\Omega_k}\frac{1}{\lambda_{M+1}^{k l}}\right) 
       a(u,u)
 \end{equation}
\end{lemma}
\begin{proof}
Using Lemma~\ref{lem:equivBF} we get 
\begin{eqnarray*}
 a(w_k,w_k)&\lesssim& \hat{a}(w_k,w_k)\\
 &=& \sum_{\tau\subset \Omega_k} \|\alpha^{1/2}\nabla w_k\|_{L^2(\tau)}^2
 +\sum_{e\subset \Omega_k} \|S_h^{1/2}[w_k]\|_{L^2(e)}^2 \\
 &+&\sum_{e\subset \partial \Omega\cap\partial \Omega_k} \|S_h^{1/2}[w_k]\|_{L^2(e)}^2
 +\sum_{e\subset \Gamma\cap\partial \Omega_k} \|S_h^{1/2}[w_k]\|_{L^2(e)}^2,
\end{eqnarray*}
where the second sum in the right hand side is over all edges of elements $\tau\subset \Omega_k$ that are not on the boundary
of $\Omega_k$.

The first three terms, when summed over $k=1,\ldots,N$, yield
\begin{eqnarray*}
&&\sum_{k=1}^N\left(\sum_{\tau\subset \Omega_k} \|\alpha^{1/2}\nabla w_k\|_{L^2(\tau)}^2
 +\sum_{e\subset \Omega_k} \|S_h^{1/2}[w_k]\|_{L^2(e)}^2 +\sum_{e\subset \partial \Omega\cap\partial \Omega_k} \|S_h^{1/2}[w_k]\|_{L^2(e)}^2\right)\\
 &&\quad =\sum_{k=1}^N\left(
 \sum_{\tau\subset \Omega_k} \|\alpha^{1/2}\nabla (u-I_0 u)\|_{L^2(\tau)}^2
 +\sum_{e\subset \Omega_k\cup(\partial \Omega\cap\partial \Omega_k)} \|S_h^{1/2}[u-I_0 u]\|_{L^2(e)}^2\right)\\
 && \quad
 \leq \hat{a}(u-I_0u,u-I_0 u), 
\end{eqnarray*}
since all terms in the left of inequality are also in the right. We can bound this term in the same way as in the proof of Lemma~\ref{lem:coarse-int-est}.
 
 It remains to bound the fourth term, that is the sum of integrals over the edges that are on $\Gamma$.
 Let $e\subset \Gamma \cap \partial \Omega_k$ be an edge on the interface $\Gamma_{k l}$, so it is the
 common edge of two triangles  $\tau_+\subset \Omega_k^{\delta}$ and $\tau_-\subset \Omega_l^{\delta}$.  Note that on $\tau_-$ the function $w_k$ is zero, hence, 
 by Lemma~\ref{lem:IPinvineq}, we get 
 \begin{eqnarray*}
  \sum_{e\subset \Gamma\cap\partial \Omega_k} \|S_h^{1/2}[w_k]\|_{L^2(e)}^2
  & \lesssim& h^{-2}\|\alpha^{1/2}(u-I_0u)\|_{L^2(\Omega_k^{\delta})}^2\\
  & \lesssim & \sum_{\Gamma_{k l} \subset \Gamma \cap \partial\Omega_k}
   h^{-2}\|\alpha^{1/2}(u-I_0u)\|_{L^2(\Gamma_{k l}^{\delta})}^2.
 \end{eqnarray*}
This term was already bounded in the proof of Lemma~\ref{lem:coarse-int-est}. Following the lines of proof there, we end the proof of the lemma here.
\end{proof}

\begin{lemma}\label{lem:stable_dec}
For any $u \in V^h$ there is a stable decomposition, that is, there exists a coarse function $u_0\in V_0$ and local functions $u_k \in V_k$ for $k=1,\ldots,N,$ such that $u=u_0+\sum_{k=1}^N u_k$, and
\begin{equation}
 a(u_0,u_0) + \sum_{k=1}^N a(u_k,u_k)\leq \left(1+\max_{\Gamma_{k l} \subset \Gamma} \frac{1}{\lambda_{M+1}^{k l}}\right)\lesssim a(u,u). \label{eq:stable-dec-bound}
\end{equation} 
\end{lemma}
\begin{proof}
%Due to Lemma~\ref{lem:equivBF} it is enough to show the last inequality for the $a(u,u)$ replaced by $\hat{a}(u,u)$ in the right-hand side of this inequality.

For $u\in V^h$, define its coarse component as the coarse interpolant of $u$, $u_0=I_0u \in V_0$. The local decomposition $u-u_0$ as given as
$$
u_k=(u-u_0)_{|\overline{\Omega}_k} \in V_k, \qquad k=1,\ldots,N,
$$ 
is defined uniquely. Accordingly, $u_k$ equals to $u-u_0$ on all triangles of $\Omega_k$ and to zero on all triangles of the remaining subdomains.
Clearly, 
$$
 u_0+\sum_{k=1}^Nu_k=u_0+\sum_{k=1}^N(u-u_0)_{|\overline{\Omega}_k}= u_0 + u - u_0=u.
$$
The stability estimate (\ref{eq:stable-dec-bound}) follows from the lemmas \ref{lem:coarse-int-est}-\ref{lem:local-est}.
\end{proof}

\section{Numerical Results}\label{sec:NumRes}
\begin{figure}[htb]
\centering
\includegraphics[height=7cm]{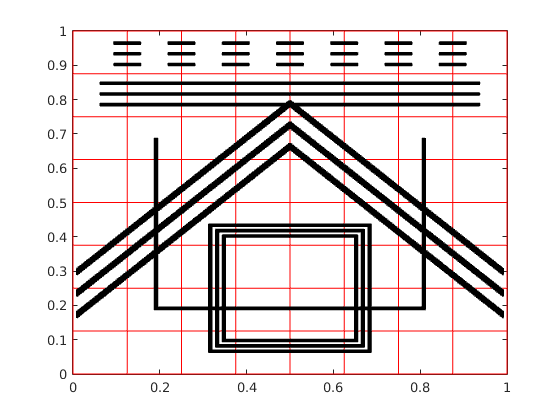}%{Ex22.jpg}
\caption{The unit square domain with $8\times 8$ square subdomains. The distribution, symmetric around $x=0.5$, has channels and inclusions (smaller channels) in a complex network. The coefficient $\alpha$ equals one on the background and $\alpha_0$ on inclusions and channels.}\label{Fig:Ex22}
\end{figure}We consider our model problem (\ref{eq:modelproblem}) to be defined on the unit square with zero Dirichlet boundary condition and a constant force function, and solve it using the symmetric interior penalty discontinuous Galerkin (SIPG) discretization (cf. \cite{Riviere:2008:DGM}) and the conjugate gradients iteration with the additive Schwarz method of this paper as the preconditioner. 
%For the Schwarz preconditioner in the two first tests, we decomposed the domain into $8\times8=64$ non-overlapping square subdomains, each with a regular triangulation of $2\times16\times16$ triangles resulting in a total of 32768 triangles. There are a total of $112$ patches, each corresponding to the interior edges, and $49$ cross points, each being connected to the four different patches meeting at the point. 
In all our experiments, the penalty parameter $\eta$ is set equal to 4, and the iteration stops as soon as the relative residual norm becomes less than $10^{-6}$. For the multiscale problem, we have chosen high contrast media which are represented by the values and jumps of the coefficient $\alpha$, as shown in the figures \ref{Fig:Ex22}--\ref{Fig:ndist}. In all these figures, $\alpha$ equals one on the background and (one or a much higher value) $\alpha_0$ on the inclusions and channels. The value of $\alpha_0$ hence defines the jump or the contrast in $\alpha$. 

As described in the paper, our coarse space has one spectral component whose basis functions are related to the first few eigenfunctions of the generalized eigenvalue problems on the patches, and one non-spectral component whose basis functions are associated with the degrees of freedom associated with the cross points. For the spectral component, a threshold is set for the eigenvalues; eigenfunctions corresponding to eigenvalues below the threshold are then chosen to construct the basis functions for the spectral component. For the non-spectral component, to keep the number of basis functions to a minimal, we include precisely one basis function per crosspoint per patch, as a consequence, two basis functions of non-spectral type per patch. It should be noted here that this rearrangement of the algorithm does not change the outcome of our theory.
%We have observed in this case a slight decrease in the condition number as opposed to including multiple basis functions per cross point per patch.
\begin{table}[htb]
\begin{center}
\begin{tabular}{|c|c|c|c|c|}
\hline
\multicolumn{1}{|c|}{} & \multicolumn{3}{|c|}{Fixed number of enrichment} & \multicolumn{1}{|c|}{Adaptive enrich.}\\
\cline{2-4}
\multicolumn{1}{|c|}{$\alpha_0$} & \multicolumn{1}{|c|}{ none } & \multicolumn{1}{|c|}{ 2 }  & \multicolumn{1}{|c|}{ 4 } & \multicolumn{1}{|c|}{$\lambda=0.18$}\\
\hline
$10^0$& $5.73\times10^1 ~(~53)$  & $1.57\times10^1 ~(31)$ & $ ~9.64 ~(24)$ & $ 15.65 ~(31)$   \\
$10^2$& $3.34\times10^2 ~(141)$  & $7.68\times10^1 ~(70)$ & $14.76 ~(33)$ & $ 22.65 ~(41)$   \\
$10^4$& $2.87\times10^4 ~(401)$  & $1.42\times10^2 ~(83)$ & $ 14.48 ~(35)$ & $ 22.67 ~(45)$   \\
$10^6$& $2.86\times10^6 ~(832)$  & $1.43\times10^2 ~(79)$ & $14.34 ~(36)$ & $ 22.66 ~(46)$  \\
\hline
\end{tabular}
\end{center}
\caption{Numerical results of the first experiment showing condition number estimates and iteration counts (in parentheses), corresponding to problem of Fig.~\ref{Fig:Ex22}. As we continue to enrich the spectral component of the coarse space with additional functions, the condition number improves, and becomes independent of the variations in $\alpha$ when all the bad eigenvalues have been used. For the adaptive enrichment $\lambda$ indicates the threshold; eigenvalues below the threshold are chosen for the enrichment.}
\label{tab:distribution}
\end{table} 
%********************************
\begin{figure}[htb]
\centering
\includegraphics[height=7cm,width=9.3cm]{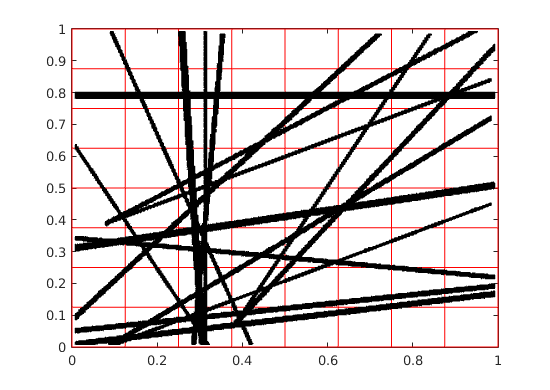}%{Ex12_h1by128.png}%{DG_dist.png}
\caption{The unit square domain with $8\times 8$ square subdomains and a complex distribution of $\alpha$ with channels crossing each other, and stretching across subdomains. The coefficient $\alpha$ equals one on the background and $\alpha_0$ on channels.}\label{Fig:Ex12}
\end{figure}

In our first two experiments, we decompose our unit square domain into $8\times8=64$ non-overlapping square subdomains, each with a regular triangulation of $2\times16\times16$ triangles resulting in a total of 32768 triangles. There are a total of $112$ patches, each corresponding to the interior edges, and $49$ cross points, each being connected to the four different patches meeting at the point. 

For the first experiment, we consider a complex distribution of the coefficient $\alpha$, symmetric around $x=0.5$, which includes both channels and inclusions (smaller channels) across subdomains, cf. Fig.~\ref{Fig:Ex22}. The results are shown in Table \ref{tab:distribution}, reporting for each test case a condition number estimate and the corresponding number of iterations required to converge. The rows in the table correspond to different magnitudes of the jumps in the coefficient. The columns correspond to different types of enrichment of the spectral component of the coarse space, either enrichment by a fixed number of eigenfunctions on every patch, see columns~1--3, or a fully adaptive enrichment where a threshold of $0.18$ for the corresponding eigenvalues have been used, see column~$4$. This choice of threshold resulted in a total of $140$ eigenfunctions for the jump $\alpha_0=1$, and $212$ eigenfunctions for the other jumps.

For the second experiment, the distribution of the coefficient is as shown in Figure~\ref{Fig:Ex12} where long channels are intersecting randomly. We have the same experimental setup as in the first experiment, with the same threshold and $\alpha$ values. The results are presented in Table \ref{tab:distribution2}. As before the rows correspond to the different magnitudes of $\alpha_0$ (the jump in $\alpha$), and the columns correspond to either fixed number of enrichment or a fully adaptive enrichment. In the case of the fully adaptive coarse space, a total of $140$ eigenfunctions are generated for $\alpha_0=1$, $227$ eigenfunctions for $\alpha_0=10^2$, and $234$ eigenfunctions for $\alpha_0\in\{10^4, 10^6\}$.

%%%%%%%%%%%%%%%%%%%%%%%%%%%%%%%%%%%%
\begin{table}[htb]
\begin{center}
\begin{tabular}{|c|c|c|c|c|}
\hline

\multicolumn{1}{|c|}{} & \multicolumn{3}{|c|}{Fixed number of enrichment} & \multicolumn{1}{|c|}{Adaptive enrich.}\\
\cline{2-4}
\multicolumn{1}{|c|}{$\alpha_0$} & \multicolumn{1}{|c|}{ none } & \multicolumn{1}{|c|}{ 2 }  & \multicolumn{1}{|c|}{ 4 } & \multicolumn{1}{|c|}{$\lambda=0.18$}\\
\hline
$10^0$& $5.73\times10^1 ~(~53)$ & $1.57\times10^1 ~(31)$ & $ ~9.64 ~(24)$ & $ 15.65 ~(31)$   \\
$10^2$& $3.18\times10^2 ~(124)$ & $1.87\times10^2 ~(80)$ & $19.17 ~(37)$  & $ 27.90 ~(44)$  \\
$10^4$& $2.98\times10^4 ~(328)$ & $1.01\times10^4 ~(122)$ & $22.04 ~(39)$ & $ 28.57 ~(47)$  \\
$10^6$& $2.98\times10^6 ~(508)$ & $2.96\times10^4 ~(131)$ & $22.07 ~(42)$ 	& $ 28.56 ~(50)$  \\
\hline
\end{tabular}
\end{center}
\caption{Numerical results of the second experiment showing condition number estimates and iteration counts (in parentheses), corresponding to the problem of Fig.~\ref{Fig:Ex12}. As we continue to enrich the spectral component of the coarse space with additional functions, the condition number improves, and becomes independent of the variations in $\alpha$ when all the bad eigenvalues have been used. For the adaptive enrichment $\lambda$ indicates the threshold; eigenvalues below the threshold are chosen for the enrichment.}
\label{tab:distribution2}
\end{table}

In both experiments above, the number of eigenfunctions used for the fixed enrichment of 4 has been larger than that required for the adaptive enrichment. This fact is reflected in their condition number estimates, as we can see it from the tables that the condition number estimates in the fixed enrichment case of 4, are smaller than those in the adaptive enrichment case. The important observation, however, is that, in the adaptive enrichment case, with a fixed $\lambda$ (an adequately chosen threshold to include the bad eigenvalues), the condition number estimates are independent of the high contrast.

%%%%%%%%%%%%%%%%%%%%%%%%%%%%%%%%%%%%%%%%%%%%%%%%%%%%%%%%%%%%%%%%%%%%%%%%%%%%%%%%%%%55
\begin{figure}[htb]
\centering
\includegraphics[height=5cm,width=8.3cm]{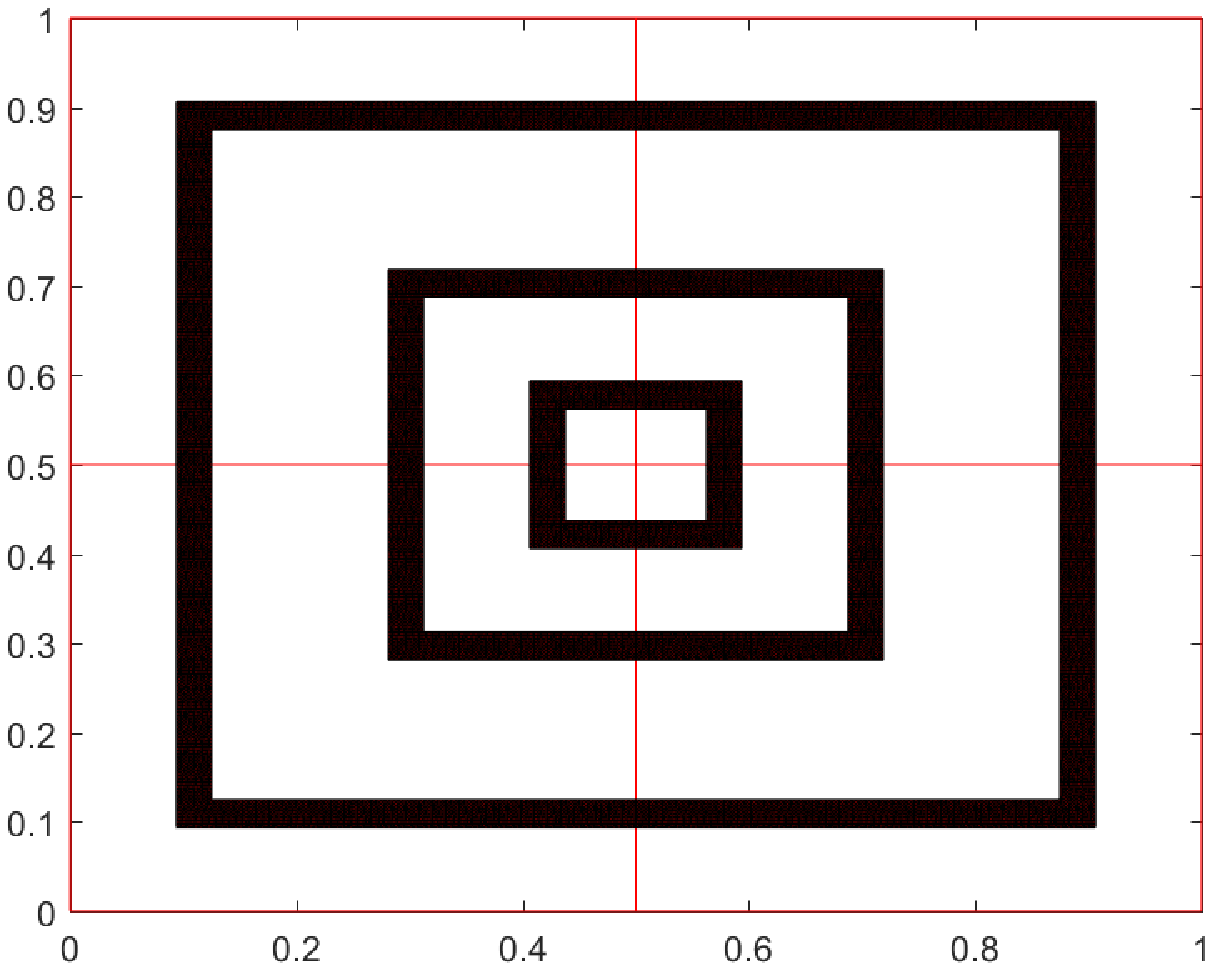}%{Ex12_h1by128.png}%{DG_dist.png}
\caption{The unit square domain decomposed into four square subdomains. The distribution of $\alpha$ (background with closed channels) is such that the eigenvalue problems on the edges are identical. The coefficient $\alpha$ equals one on the background and $\alpha_0$ on the channels.}\label{Fig:ndist}
\end{figure} 
%%%%%%%%%%%%%%%%%%%%%%%%%%%%%%%%%%%%%%%%%
\begin{table}[htp]
\begin{center}
\begin{tabular}{|c|c|c|c|}
\hline
Eigenvalue
&$h=1/32$ 	& $h=1/64$ 	&  $h=1/128$ 	\\
  \hline	
	$\lambda_1$
	&$8.67\times10^{-8}$	&$2.29\times10^{-8}$
	&$5.89\times10^{-9}$\\
	$\lambda_2$
	&$5.23\times10^{-7}$	&$1.37\times10^{-7}$
	&$3.50\times10^{-8}$\\
	$\lambda_3$
	&$1.32\times10^{-6}$	&$3.39\times10^{-7}$
	&$8.66\times10^{-8}$\\
   \hline
	$\lambda_4$
	&$0.3690$		
	&$0.0916$			
	&$0.0229$		\\
	$\lambda_5$
	&${\bf 0.5327}$
	&$0.1320$		
	&$0.0329$		\\
	$\lambda_6$
	&% $1.5110$		
	&$0.3686$			
	&$0.0916$		\\
	$\lambda_7$
	&% $2.2061$
	&${\bf 0.5327}$		
	&$0.1320$		\\
	$\lambda_8$
	&% $3.4760$		
	&% $0.7023$			
	&$0.1808$		\\
	$\lambda_9$
	&% $3.5332$
	&% $0.8379$		
	&$0.2066$		\\
    $\lambda_{10}$
    &% $3.5353$
	&% $1.2163$		
	&$0.2981$		\\
	$\lambda_{11}$
	&% $3.5353$
	&% $1.5190$		
	&$0.3638$		\\
	$\lambda_{12}$
	&% $3.5353$
	&% $2.2061$		
	&${\bf 0.5327}$		\\
	\hline
\end{tabular}
\end{center}
\caption{Smallest eigenvalues corresponding to the generalized eigenvalue problem on each patch of the problem of Fig.\ref{Fig:ndist} for the three different mesh sizes. The eigenvalues in boldface indicate the first eigenvalue above the threshold $\lambda=0.415$.} 
\label{tab:eigenvalues}
\end{table}
%%%%%%%%%%%%%%%%%%%%%%%%%%%%%%%%%%%%%%%%
%\begin{figure}[htb]
%\centering
%\includegraphics[height=7cm,width=9.3cm]{eigdist.eps}
%\caption{Distribution of eigenvalues. Number of eigenvalues along the x-axis and numerical value of eigenvalues on the y-axis.}\label{Fig:eigdist}
%\end{figure} 
%%%%%%%%%%%%%%%%%%%%%%%%%%%%%%%%%%%%

For our third experiment, we choose a smaller decomposition of the domain, with $2\times 2$ subdomains, giving a total of four patches, and a $\alpha$ distribution which is symmetric in a way such that the eigenvalue problems on the four patches become identically the same, cf. Fig.~\ref{Fig:ndist}. 
%This distribution is symmetric in a way so that the eigenvalue problems on the four patches become identically the same. 
Eigenvalues below the threshold value of $\lambda=0.415$ are shown in Table~\ref{tab:eigenvalues}, for $h=1/32$ and its two levels of refinement, and $\alpha_0=10^6$. We observe that the three channels here result in three eigenvalues (the bad eigenvalues) several magnitudes lower than the remaining eigenvalues. It has also been observed in other papers, that there is a direct correlation between the number of bad eigenvalues and the number of channels with large coefficients. For the adaptive coarse space case, we set a threshold so that only a small number of eigenvalues are below the threshold (enough to include all the bad eigenvalues), e.g., $\lambda=0.415$ is enough to include the bad eigenvalues plus one eigenvalue for the starting mesh size $h=1/32$. The eigenvalues in boldface, in Table~\ref{tab:eigenvalues}, indicate the next lowest eigenvalue above the threshold of $\lambda=0.415$. 

In Table~\ref{tab:distribution3} we present the numerical results for the third experiment. The rows in the table correspond to the mesh size parameter $h$ and the jump in the coefficient $\alpha$, and the columns correspond to the fully adaptive enrichment for different values of threshold $\lambda$. For each test, a condition number estimate and its corresponding number of enrichment per patch (inside square brackets) are given. As seen from the table, for a fixed $\lambda$, the condition number estimates remain close to each other as the mesh size $h$ varies,  indicating that the condition number is independent of the mesh size parameter. Although, in this case, smaller mesh size implies more eigenfunctions to be included in the construction of the coarse space, the number of which appears to vary inversely with the mesh size, cf. Table~\ref{tab:eigenvalues}. Further, as we compare the condition number estimates for different values of $\alpha_0$ (the contrast) in the table, once again, we see that the number is independent of the contrast for fixed $\lambda$. Finally, to see how the condition number bound depends on $\lambda_{M+1}$ (i.e., the next lowest eigenvalue above the threshold $\lambda$), we refer to Fig.~\ref{Fig:condplot} where condition number estimates are plotted against the inverse of $\lambda_{M+1}$, for $\lambda=0.208$ in the Table~\ref{tab:distribution3}, indicating that the condition number is inversely proportional to $\lambda_{M+1}$.

The numerical results of this section are consistent with the theory. All our experiments demonstrate the effectiveness of our algorithm for problems with highly varying coefficients, showing it is robust with respect to the contrast as well as the distribution, in general, requiring only a modest number of spectral enrichment of the coarse space. This enrichment is based on solving positive definite generalized eigenvalue problems on thin patches, each covering a subdomain interface. The eigenvalue problems are therefore relatively small sized. Solving eigenvalue problem on a patch and then extending its eigenfunctions inside subdomains are two separate routines in our framework, which makes our algorithm less intricate, and more susceptible of code reuse.

%%%%%%%%%%%%%%%%%%%%%%%%%%%%%%%%%%%%

\begin{table}[htb]
\begin{center}
\begin{tabular}{|l|l|c|c|c|c|}
\hline
\multicolumn{1}{|c|}{} &\multicolumn{1}{|c|}{} & \multicolumn{1}{|c|}{Enrichment:} & \multicolumn{3}{|c|}{Adaptive enrichment}\\
%\cline{3-3}
\multicolumn{1}{|c|}{$\alpha_0$} &\multicolumn{1}{|c|}{$h$} & \multicolumn{1}{|c|}{ none} & \multicolumn{1}{|c}{ $\lambda=0.208$ }  & \multicolumn{1}{c}{$\lambda=0.415$ } &
\multicolumn{1}{c|}{$\lambda=0.830$}\\
\hline
         &$1/32$& $6.02\times10^2 $ & $10.82 \quad [03]$ & $12.00 \quad [04]$ & $ ~7.61 \quad [05]$   \\
$10^2$  & $1/64$& $1.20\times10^3 $ & $14.80 \quad [05]$ & $11.79 \quad [06]$ & $ 10.21 \quad [08]$  \\
        &$1/128$& $2.39\times10^3 $ & $17.09 \quad [09]$ & $11.69 \quad [11]$ & $ ~9.93 \quad [16]$  \\
        &$1/256$& $4.77\times10^3 $ & $20.31 \quad [16]$ & $12.73 \quad [23]$ & $ 9.54 \quad [32]$  \\
\hline
        & $1/32$& $5.55\times10^4 $ & $13.06 \quad [03]$ & $12.51 \quad [04]$ & $ ~7.64 \quad [05]$   \\
$10^4$  & $1/64$& $1.11\times10^5 $ & $15.65 \quad [05]$ & $12.88 \quad [06]$ & $ 10.20 \quad [08]$  \\
        &$1/128$& $2.22\times10^5 $ & $17.07 \quad [09]$ & $12.88 \quad [11]$ & $ ~9.73 \quad [14]$  \\
        &$1/256$& $4.45\times10^5 $ & $17.68 \quad [19]$ & $13.98 \quad [23]$ & $ 9.83 \quad [30]$  \\
\hline
        & $1/32$& $5.54\times10^6 $ & $13.11 \quad [03]$ & $13.06 \quad [04]$ & $ ~7.92 \quad [05]$   \\
$10^6$  & $1/64$& $1.11\times10^7 $ & $15.68 \quad [05]$ & $13.09 \quad [06]$ & $ 11.03 \quad [08]$  \\
        &$1/128$& $2.22\times10^7 $ & $17.31 \quad [09]$ & $13.15 \quad [11]$ & $ 10.67 \quad [14]$  \\
        &$1/256$& $4.45\times10^7 $ & $17.83 \quad [19]$ & $14.13 \quad [23]$ & $ 10.67 \quad [30]$  \\
\hline
\end{tabular}
\end{center}
\caption{Numerical results of the third experiment, corresponding to Fig.~\ref{Fig:ndist}, showing condition number estimates and their corresponding numbers of enrichment (in brackets) for different values of $\lambda$ (threshold for the smallest eigenvalues) for the adaptive enrichment.}
\label{tab:distribution3}
\end{table}

\begin{figure}[!ht]
 \centerline{
 %\SetLabels
 %\L\B (.28*.805)  $h=1/256$ \\
 %\L\B (.28*.755)  $h=1/128$ \\
 %\L\B (.28*.705)  $h=1/64$ \\
 %\endSetLabels
 %\ShowGrid
 \AffixLabels{
 \includegraphics[height=6cm,width=7cm]{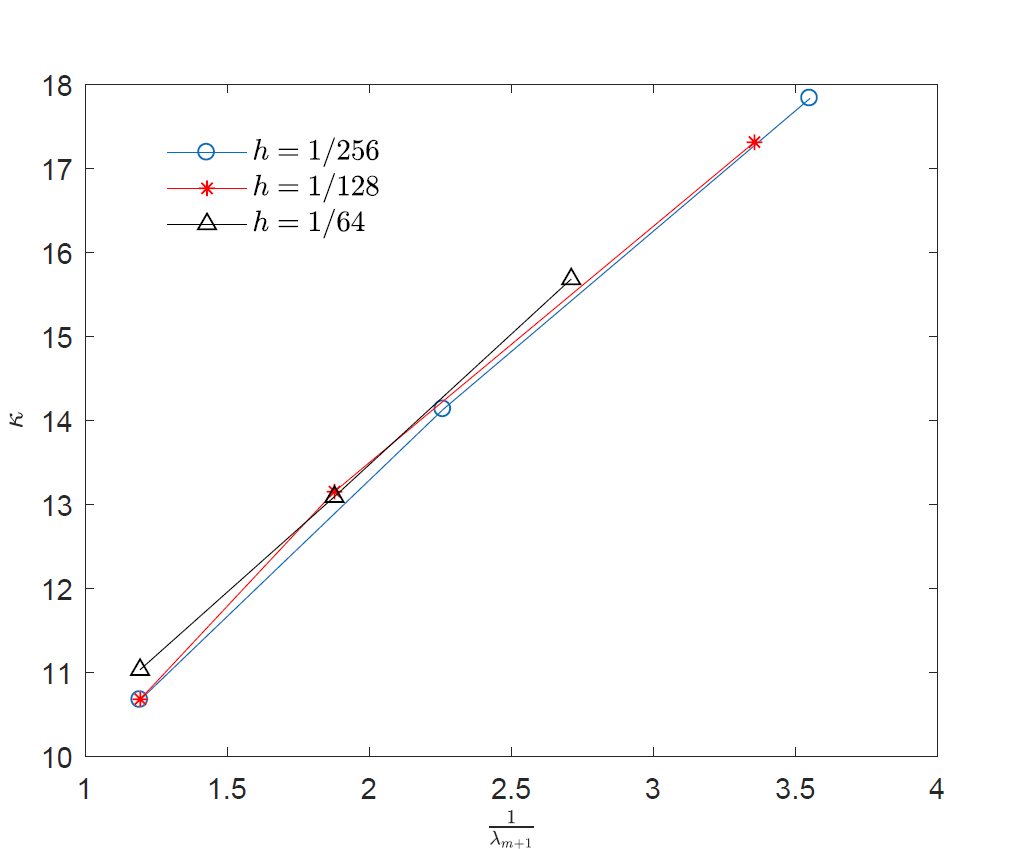}
 }}
 \caption{Plot of condition number estimates $\kappa$ (along the $y$-axis) against ${1}/{\lambda_{m+1}}$ (along the $x$-axis), for the threshold $\lambda=0.208$ and contrast $\alpha_0=10^6$, cf. Table \ref{tab:distribution3}. $\lambda_{m+1}$ is the next smallest eigenvalue above the threshold.}\label{Fig:condplot}
\end{figure}

% Our experiments have shown that the algorithm can handle problems with high contrast and complex distribution (a network of channels and inclusions) with a modest number of spectral enrichment, demonstrating the power of our method. 

%We close with some concluding remarks regarding the numerical results. From Test~$1$ and Test~$2$, it is clear that our proposed method can handle challenging distributions and that the preconditioned linear systems have a low condition number even for distributions with significant jump in coefficient value. From Table~\ref{tab:eigenvalues} we observe that the lowest eigenvalues $\lambda$ of the eigenvalue problems on the patches are of order $(h/H)^2$. Applying this to Theorem~\ref{thm:Cond_est} would indicate that our method has a condition number of order $(H/h)^2$. The results of Test~$3$ indicate a linear dependence on $(H/h)$. 
%For the adaptive coarse space with a fixed eigenvalue threshold in Test~$3$, we see that the condition number is constant as Theorem~\ref{thm:Cond_est} would suggest. This result indicates that there is no $h/H$ factor in the general constant $C$ that we have missed. Instead, we see an increasing number of new eigenfunctions needed to maintain a constant condition number under linear mesh refinement. Based on these observations it might be proper to regard Theorem~\ref{thm:Cond_est} as an estimate for the adaptive coarse space with a fixed threshold, and that a constant condition number can be expected as long as the threshold remains fixed.

%\newpage
%\clearpage

\section*{Acknowledgments}
Leszek Marcinkowski  was partially supported by Polish Scientific Grant: National Science Center:  2016/21/B/ST1/00350.

%\bibliographystyle{siam}
%\bibliography{refdg}

\end{document}